\documentclass{amsart}
\usepackage{amsfonts, amsbsy, amsmath, amssymb}

\makeatletter
\@namedef{subjclassname@2010}{%
  \textup{2020} Mathematics Subject Classification}
\makeatother

\ifx\fiverm\undefined 
  \newfont\fiverm{cmr5} 
\fi 

\input prepictex.tex
\input pictex.tex
\input postpictex.tex

\newtheorem{thm}{Theorem}[section]
\newtheorem{lem}[thm]{Lemma}

\newtheorem{thm-con}[thm]{Theorem-Conjecture}
\numberwithin{equation}{section}

\theoremstyle{definition}

\allowdisplaybreaks

\newcommand{\f}{\Bbb F}

\begin{document}

\title[a conjecture on permutation rational functions]{on a conjecture on permutation rational functions over finite fields}

\author[Daniele Bartoli]{Daniele Bartoli}
\address{Dipartimento di Matematica e Informatica, Universit\`a degli Studi di Perugia, Italy} 
\email{daniele.bartoli@unipg.it}

\author[Xiang-dong Hou]{Xiang-dong Hou}
\address{Department of Mathematics and Statistics,
University of South Florida, Tampa, FL 33620, USA}
\email{xhou@usf.edu}


\keywords{finite field, Lang-Weil bound, permutation, rational function}

\subjclass[2010]{11R58, 11T06, 11T55, 14H05}

\begin{abstract} 

Let $p$ be a prime and $n$ be a positive integer, and consider $f_b(X)=X+(X^p-X+b)^{-1}\in \f_p(X)$, where $b\in\f_{p^n}$ is such that $\text{Tr}_{p^n/p}(b)\ne 0$. It is known that (i) $f_b$ permutes $\f_{p^n}$ for $p=2,3$ and all $n\ge 1$; (ii) for $p>3$ and $n=2$, $f_b$ permutes $\f_{p^2}$ if and only if $\text{Tr}_{p^2/p}(b)=\pm 1$; and (iii) for $p>3$ and $n\ge 5$, $f_b$ does not permute $\f_{p^n}$. It has been conjectured that for $p>3$ and $n=3,4$, $f_b$ does not permute $\f_{p^n}$. We prove this conjecture for sufficiently large $p$.
 
\end{abstract}

\maketitle

\section{Background}

Let $\f_q$ denote the finite field with $q$ elements. Polynomials over $\f_q$ that permute $\f_q$, called {\em permutation polynomials} (PPs) of $\f_q$, have been extensively studied in the theory and applications of finite fields. Recently, permutation rational functions (PRs) of finite fields also attracted considerable attention. There are a number of reasons for studying PRs. Certain types of PPs of high degree can be reduced to PRs of low degree; this approach has allowed people to solve numerous questions about PPs \cite{Bartoli-FFA-2018, Bartoli-FFA-2020, Cao-Hou-Mi-Xu-FFA-2020, Hou-CC-2019, Hou-arXiv1906.07240, Hou-Tu-Zeng-FFA-2020, Li-Qu-Li-Fu-AA-2018, Tu-Zeng-FFA-2018, Tu-Zeng-Li-Helleseth-FFA-2018, Wang-2019, Zieve-PAMS-2009}. Oftentimes, PRs reveal phenomena that are not present in PPs; understanding these phenomena requires methods that are different from those in traditional approaches to PPs.

This paper concerns a conjecture on PRs of the type
\[
f_b(X)=X+\frac 1{X^p-X+b}\in\f_p(X)
\]
of $\f_{p^n}$, where $p$ is a prime, $n$ is a positive integer, and $b\in\f_{p^n}$ is such that $\text{Tr}_{p^n/p}(b)\ne 0$. In \cite{Yuan-Ding-Wang-Pieprzyk-FFA-2008}, Yuan et al. proved that for $p=2,3$ and all $n\ge 1$, $f_b$ is a PR of $\f_{p^n}$. Recently, it was shown in \cite{Hou-Sze-arXiv:1910.11989} that for $p>3$ and $n\ge 5$, $f_b$ is not a PR of $\f_{p^n}$, and for $p>3$ and $n=2$, $f_b$ is a PR of $\f_{p^2}$ if and only if $\text{Tr}_{p^2/p}(b)=\pm 1$. Based on computer search, it was conjectured in \cite{Hou-Sze-arXiv:1910.11989} that for $p>3$ and $n=3,4$, $f_b$ is not a PR of $\f_{p^n}$. We will prove this conjecture for sufficiently large $p$. Our approach relies on the Lang-Weil bound on the number of zeros of absolutely irreducible polynomials over finite  fields. The main technical ingredient of our proof is a claim that that a certain polynomial of degree 18 in $\f_p[Y_1,Y_2,Y_3]$ has a cyclic absolutely irreducible factor in $\f_p[Y_1,Y_2,Y_3]$ and a claim that a certain polynomial of degree 46 in $\f_p[Y_1,Y_2,Y_3,Y_4]$ has a cyclic absolutely irreducible factor in $\f_p[Y_1,Y_2,Y_3,Y_4]$.

Throughout the paper, $\overline\f_q$ denotes the algebraic closure of $\f_q$. For $f\in\f_q[X_1,\dots,X_n]$, define 
\[
V_{\f_q^n}(f)=\{(x_1,\dots,x_n)\in\f_q^n:f(x_1,\dots,x_n)=0\}.
\]
$f$ is said to be {\em absolutely irreducible} if it is irreducible in $\overline\f_p[X_1,\dots,X_n]$. The resultant of two polynomials $f(X)$ and $g(X)$ in $X$ is denoted by $\text{Res}(f,g;X)$.

\section{Cyclic Shift and the Forbenius}

For $\sigma\in\text{Aut}(\f_q)$ and $f\in\f_q[X_1,\dots,X_n]$, let $\sigma(f)$ denote the resulting polynomial by applying $\sigma$ to the coefficients of $f$; this defines an action of $\text{Aut}(\f_q)$ on $\f_q[X_1,\dots, X_n]$. Let $\rho$ be the cyclic shift on the indeterminates $X_1,\dots,X_n$: $\rho(X_1,\dots,X_n)=(X_2,,X_3,\dots,X_n, X_1)$. For $f\in\f_q[X_1,\dots,X_n]$ and $\rho^i\in\langle\rho\rangle$, let $f^{\rho^i}=f(\rho^i(X_1,\dots,X_n))$; this gives an action of $\langle\rho\rangle$ on $\f_q[X_1,\dots,X_n]$. A polynomial $f\in\f_q[X_1,\dot,X_n]$ is called {\em cyclic} if $f^\rho=f$ and is called {\em pseudo-cyclic} if $f^\rho=cf$ for some $n$th unity in $\f_q$.

For $z\in\f_{q^n}$, $z,z^q,\dots,z^{q^{n-1}}$ form a normal basis of $\f_{q^n}$ over $\f_q$ if and only if the Moore matrix of $z$,
\[
M(z)=\left[
\begin{matrix}
z&z^q&\cdots&z^{q^{n-1}}\cr
z^q&z^{q^2}&\cdots&z\cr
\vdots&\vdots&&\vdots\cr
z^{q^{n-1}}&z&\cdots&z^{q^{n-2}}
\end{matrix}\right],
\]
is invertible. An $n\times n$ matrix $A$ over $\f_{q^n}$ is of the form $M(z)$ for some $z\in\f_{q^n}$ if and only if $\sigma(A)=CA=AC^{-1}$, where $\sigma(\ )=(\ )^q$ is the Frobenius map of $\f_{q^n}/\f_q$, $\sigma(A)$ is the result of entry-wise action of $\sigma$ on $A$, and 
\[
C=\left[
\begin{matrix}
0&1&0&\cdots&0\cr
0&0&1&\cdots&0\cr
\vdots&\vdots&\vdots&\ddots&\vdots\cr
0&0&0&\cdots&1\cr
1&0&0&\cdots&0
\end{matrix}\right].
\]
From this, it is easy to see that if $M(z)$ is invertible, then $M(z)^{-1}=M(w)$ for some $w\in\f_{q^n}$.

\begin{lem}\label{L2.1}
Let $z\in\f_{q^n}$ be such that $\det M(z)\ne 0$. Let $f\in\overline\f_q[X_1,\dots,X_n]$ and $g=f((X_1,\dots,X_n)M(z))$. Then
\begin{itemize}
\item[(i)] $f$ is cyclic if and only if $g$ is cyclic;
\item[(ii)] $f\in\f_q[X_1,\dots,X_n]$ and is cyclic if and only if $g\in\f_q[X_1,\dots,X_n]$ and is cyclic.
\end{itemize}
\end{lem}

\begin{proof}
Since $f=g((X_1,\dots,X_n)M(w))$, where $M(w)=M(z)^{-1}$, we only have to prove the ``only if'' part in both (i) and (ii).

\medskip
(i) ($\Rightarrow$) We have 
\begin{align*}
g((X_1,\dots,X_n)C)\,&=f((X_1,\dots,X_n)C M(z))\cr
&=f((X_1,\dots,X_n) M(z)C^{-1})\cr
&=f((X_1,\dots,X_n) M(z))\kern5em\text{(since $f$ is cyclic)}\cr
&=g.
\end{align*}

(ii) ($\Rightarrow$) We have 
\begin{align*}
\sigma(g)\,&=\sigma(f((X_1,\dots,X_n) M(z)) )\cr
&=f((X_1,\dots,X_n) \sigma(M(z)))\kern5em\text{(since $f\in\f_q[X_1,\dots,X_n]$)}\cr
&=f((X_1,\dots,X_n) M(z)C^{-1})\cr
&=g.
\end{align*}
\end{proof}

\section{The Case $n=3$}

For $n=3$, we have the following theorem.

\begin{thm}\label{T3.1}
Let $p\ge 1734097$ be a prime and $b\in\f_{p^3}$ be such that $\text{\rm Tr}_{p^3/p}(b)\ne 0$. Then $f_b$ is not a PR of $\f_{p^3}$. (Note: $1734097$ is the $130492$th prime.)
\end{thm}

\begin{proof}
We have
\[
f_b(X+Y)-f_b(X)=\frac{Y(Z(X)^2+(Y^p-Y)Z(X)+1-Y^{p-1})}{Z(X)Z(X+Y)},
\]
where
\[
Z(X)=X^p-X+b.
\]
Let 
\[
F(X,Y)=Z(X)^2+(Y^p-Y)Z(X)+1-Y^{p-1}\in\f_{p^3}[X,Y].
\]
It suffices to show that there exists $(x,y)\in\f_{p^3}^2$ with $y\ne 0$ such that $F(x,y)=0$, i.e., such that
\begin{equation}\label{3.1}
z^2+(y^p-y)z+1-y^{p-1}=0,
\end{equation}
where $z=x^p-x+b$. The solution of \eqref{3.1} for $z$ is
\begin{equation}\label{3.2}
z=\frac 12(-(y^p-y)+\Delta),
\end{equation}
where
\begin{align}\label{3.3}
\Delta^2\,&=(y^p-y)^2-4(1-y^{p-1})\\
&=(y^{p-1}-1)(y^2(y^{p-1}-1)+4)\cr
&=y^{2p}+y^2-2y^{1+p}+4y^{p-1}-4.\nonumber
\end{align}
Therefore, it suffices to show that there exist $\Delta,y\in\f_{p^3}$, with $y\ne 0$ and $\text{Tr}_{p^3/p}(\Delta)=2\text{Tr}_{p^3/p}(b)=:t$, satisfying
\begin{equation}\label{3.4}
\Delta^2=y^{2p}+y^2-2y^{1+p}+4y^{p-1}-4.
\end{equation}
Assume for the time being that we already have $\Delta,y\in\f_{p^3}$, with $y\ne 0$ and $\text{Tr}_{p^3/p}(\Delta)=t$, that satisfy \eqref{3.4}. Let $y=y_1,y_2=y^p, y_3=y^{p^2}$ and $\Delta_1=\Delta, \Delta_2=\Delta^p, \Delta_3=\Delta^{p^2}$. Then we have
\begin{gather}\label{3.5}
\Delta_1^2=y_1^2+y_2^2-2y_1y_2+4\frac{y_2}{y_1}-4,\\ \label{3.6}
\Delta_2^2=y_2^2+y_3^2-2y_2y_3+4\frac{y_3}{y_2}-4,\\ \label{3.7}
\Delta_3^2=y_3^2+y_1^2-2y_3y_1+4\frac{y_1}{y_3}-4, \\ \label{3.8}
\Delta_1+\Delta_2+\Delta_3=t.
\end{gather}
From \eqref{3.8} we can express $\Delta_1$ in terms of $\Delta_1^2$, $\Delta_2^2$, $\Delta_3^2$ and $t$ through the following calculation:
\begin{gather}
(t-\Delta_1)^2=(\Delta_2+\Delta_3)^2,\cr
t^2+\Delta_1^2-\Delta_2^2-\Delta_3^2-2t\Delta_1=2\Delta_2\Delta_3,\cr
(t^2+\Delta_1^2-\Delta_2^2-\Delta_3^2)^2+4t^2\Delta_1^2-4t\Delta_1(t^2+\Delta_1^2-\Delta_2^2-\Delta_3^2)=4\Delta_2^2\Delta_3^2,\cr
\label{3.9}
\Delta_1=\frac{(t^2+\Delta_1^2-\Delta_2^2-\Delta_3^2)^2+4t^2\Delta_1^2-4\Delta_2^2\Delta_3^2 } {4t(t^2+\Delta_1^2-\Delta_2^2-\Delta_3^2)},
\end{gather}
provided the denominator is nonzero. Using \eqref{3.5} -- \eqref{3.7}, we can write \eqref{3.9} as
\begin{equation}\label{3.11}
\Delta_1=\frac{P(y_1,y_2,y_3)}{4ty_1y_2y_3Q(y_1,y_2,y_3)},
\end{equation}
where
\begin{equation}\label{3.12}
P(Y_1,Y_2,Y_3)=16Y_1^4Y_2^2+32Y_1^3Y_2^2Y_3-\cdots-16Y_1Y_2^3Y_3^4,
\end{equation}
\begin{equation}\label{3.13}
Q(Y_1,Y_2,Y_3)=-4Y_1^2Y_2+4Y_1Y_2Y_3+\cdots-2Y_1Y_2Y_3^3.
\end{equation}
It follows that 
\begin{gather}\label{3.14}
\Delta_2=\frac{P(y_2,y_3,y_1)}{4ty_1y_2y_3Q(y_2,y_3,y_1)},\\
\label{3.15}
\Delta_3=\frac{P(y_3,y_1,y_2)}{4ty_1y_2y_3Q(y_3,y_1,y_2)}.
\end{gather}
Using \eqref{3.11}, equation \eqref{3.5} becomes
\begin{equation}\label{3.17}
\frac{G(y_1,y_2,y_3)}{16t^2 y_1^2y_2^2y_3^2Q(y_1,y_2,y_3)^2}=0,
\end{equation}
and using \eqref{3.11}, \eqref{3.14} and \eqref{3.15}, equation \eqref{3.8} becomes
\begin{equation}\label{3.16}
\frac{G(y_1,y_2,y_3)}{4ty_1y_2y_3Q(y_1,y_2,y_3)Q(y_2,y_3,y_1)Q(y_3,y_1,y_2)}
=0,
\end{equation}
where $G(Y_1,Y_2,Y_3)\in\f_p[Y_1,Y_2,Y_3]$ is a cyclic polynomial of degree 18. More precisely,
\begin{equation}\label{3.18}
G=g_{18}+g_{16}+g_{14}+g_{12},
\end{equation}
where $g_i\in\f_p[Y_1,Y_2,Y_3]$ is homogeneous of degree $i$:
\begin{align}\label{3.19}
&g_{18}=-64t^2Y_1^4Y_2^4Y_3^4(Y_1-Y_2)^2(Y_2-Y_3)^2(Y_3-Y_1)^2,\\ \label{3.20}
&g_{16}=16Y_1^2Y_2^2Y_3^2(16Y_1^6Y_2^4-32Y_1^5Y_2^5+\cdots+16Y_2^4Y_3^6),\\ \label{3.21}
&g_{14}=8Y_1Y_2Y_3(64Y_1^7Y_2^4-64Y_1^6Y_2^5-\cdots-64Y_1^2Y_2^2Y_3^7),\\ \label{3.22}
&g_{12}=256Y_1^8Y_2^4+1024Y_1^7Y_2^4Y_3-\cdots +256Y_1^4Y_3^8.
\end{align}

We will show that there exists $y\in\f_{p^3}$ such that $G(y,y^p,y^{p^2})=0$ but $Q(y,y^p,y^{p^2})\ne 0$. Once this is done, the proof of the theorem is completed as follows: Clearly, $y\ne 0$. Let $y_1=y,y_2=y^p, y_3=y^{p^2}$ and let $\Delta_1,\Delta_2,\Delta_3$ be given by \eqref{3.11}, \eqref{3.14} and \eqref{3.15}. Then \eqref{3.5} and \eqref{3.8} hold. Let $\Delta=\Delta_1$. By \eqref{3.8}, $\text{Tr}_{p^3/p}(\Delta)=t$; by \eqref{3.5}, $\Delta$ and $y$ satisfy \eqref{3.4}. 

Choose $z\in\f_{p^3}$ such that $M(z)$ is invertible and let
\[
G_1=G((Y_1,Y_2,Y_3)M(z)).
\]
By Lemma~\ref{L3.2} below, $G$ has a cyclic absolutely irreducible factor $d\in\f_p[Y_1,Y_2,Y_3]$. Let $d_1=d((Y_1,Y_2,Y_3)M(z))$. Then $d\mid G$ implies that $d_1\mid G_1$, Lemma~\ref{L2.1} implies that $d_1\in\f_p[Y_1,Y_2,Y_3]$ and is cyclic, and the absolute irreducibility of $d$ implies the absolute irreducibility of $d_1$. The Lang-Weil bound \cite{Lang-Weil-AJM-1954} states that
\[
|V_{\f_p^3}(d_1)|=p^2+O(p^{3/2})\qquad \text{as}\ p\to\infty.
\]
More precisely, by \cite[Theorem~5.2]{Cafure-Matera-FFA-2006},
\begin{equation}\label{3.23}
|V_{\f_p^3}(d_1)|\ge p^2-(18-1)(18-2)p^{3/2}-5\cdot 18^{13/3}p=p^2-272p^{3/2}-5\cdot 18^{13/3}p.
\end{equation}
We find that 
\[
\text{Res}(G,Q;Y_3)=2^{16}Y_1^{14}Y_2^8(-256Y_1^4-\cdots-8t^2Y_1^4Y_2^6)^2(256Y_1^3+\cdots+4t^2Y_1^2Y_2^7)^2\ne 0.
\]
Hence $\text{gcd}(G,Q)=1$. Let $Q_1=Q((Y_1,Y_2,Y_3)M(z))$. It follows that $\text{gcd}(G_1,Q_1)=1$. By \cite[Lemma~2.2]{Cafure-Matera-FFA-2006},
\begin{equation}\label{3.24}
|V_{\f_p^3}(G_1)\cap V_{\f_p^3}(Q_1)|\le 18^2p.
\end{equation}
Therefore,
\begin{align*}
|V_{\f_p^3}(G_1)\setminus V_{\f_p^3}(Q_1)|\,&=|V_{\f_p^3}(G_1)|-|V_{\f_p^3}(G_1)\cap V_{\f_p^3}(Q_1)|\cr
&\ge |V_{\f_p^3}(d_1)|-18^2p\cr
&\ge p^2-272p^{3/2}-(5\cdot 18^{13/3}+18^2)p\cr
&=p(p-272p^{1/2}-(5\cdot 18^{13/3}+18^2) )\cr
&>0
\end{align*}
since
\[
p\ge 1734097>1734081\approx\frac 14\bigl[271+(272^2+4(5\cdot 18^{13/3}+18^2))^{1/2} \bigr]^2.
\]
Let $(a_1,a_2,a_3)\in V_{\f_p^3}(G_1)\setminus V_{\f_p^3}(Q_1) $ and let $y=a_1z+a_2z^q+a_3z^{q^2}\in\f_{p^3}$. Then $G(y,y^p,y^{p^2})=0$ but $Q(y,y^p,y^{p^2})\ne 0$. The proof is complete.
\end{proof}

\begin{lem}\label{L3.2}
The polynomial $G$ in \eqref{3.18} has a cyclic absolutely irreducible factor in $\f_p[Y_1,Y_2,Y_3]$.
\end{lem}

\begin{proof}
Let $\rho$ denote the cyclic shift $(Y_1,Y_2,Y_3)\mapsto(Y_2,Y_3,Y_1)$ and let $\sigma\in\text{Aut}(\overline\f_p)$ be the Frobenius map $(\ )\mapsto(\ )^p$. Recall that the homogeneous component of the highest degree of $G$ is 
\[
g_{18}=-64t^2Y_1^4Y_2^4Y_3^4(Y_1-Y_2)^2(Y_2-Y_3)^2(Y_3-Y_1)^2.
\]
All pseudo-cyclic factors of $g_{18}$ in $\overline\f_p[Y_1,Y_2,Y_3]$ are cyclic; therefore, all pseudo-cyclic factors of $G$ in $\overline\f_p[Y_1,Y_2,Y_3]$ are cyclic.

\medskip
$1^\circ$ We have
\[
G=a_8(Y_1,Y_2)Y_3^8+a_7(Y_1,Y_2)Y_3^7+\cdots,
\]
where
\begin{equation}\label{a8}
a_8(Y_1,Y_2)=16Y_1^2\alpha(Y_1,Y_2,t)\alpha(Y_1,Y_2,-t)
\end{equation}
and
\[
\alpha(Y_1,Y_2,T)=4Y_1+4TY_1Y_2+4Y_1^2Y_2+T^2Y_1Y_2^2+2TY_1^2Y_2^2-4Y_2^3-2TY_1Y_2^3.
\]
We claim that $\alpha(Y_1,Y_2,t)$ and $\alpha(Y_1,Y_2,-t)$ are irreducible in $\overline\f_p[Y_1,Y_2]$. The discriminant of $\alpha(Y_1,Y_2,t)$, as a polynomial in $Y_1$, is 
\[
D=16+32tY_2+24t^2 Y_2^2 -16tY_2^3 +8t^3 Y_2^3 +64Y_2^4 -16t^2 Y_2^4 +t^4 Y_2^4 +32tY_2^5 -4t^3 Y_2^5 +4t^2 Y_2^6.
\]
Assume to the contrary that $D$ is a square in $\overline\f_p[Y_2]$. Then 
\[
D-(2tY_2^3+aY_2^2+bY_2+c)^2=0,
\]
where $a,b\in\overline\f_p$ and $c=\pm 4$. 

When $c=4$,
\[
D-(2tY_2^3+aY_2^2+bY_2+4)^2\equiv -8(b-4t)Y_2 \pmod{Y_2^2}.
\]
Then $b=4t$, and it follows that 
\begin{align*}
&D-(2tY_2^3+aY_2^2+bY_2+4)^2\cr
&=-8(a-t^2)Y_2^2+8t(-4-a+t^2)Y_2^3+(64-a^2-32t^2+t^4)Y_2^4-4t(-8+a+t^2)Y_2^5\cr
&\ne 0,
\end{align*}
which is a contradiction.

When $c=-4$,
\[
D-(2tY_2^3+aY_2^2+bY_2-4)^2\equiv 8(b+4t)Y_2 \pmod{Y_2^2}.
\]
Then $b=-4t$, and it follows that 
\begin{align*}
&D-(2tY_2^3+aY_2^2+bY_2-4)^2\cr
&=8(a+t^2)Y_2^2+8t(a+t^2)Y_2^3+(64-a^2+t^4)Y_2^4-4t(-8+a+t^2)Y_2^5\cr
&\ne 0,
\end{align*}
which is a contradiction.

\medskip
$2^\circ$ Write $\alpha_1=\alpha(Y_1,Y_2,t)$ and $\alpha_2=\alpha(Y_1,Y_2,-t)$. Let $f,h\in\overline\f_p[Y_1,Y_2,Y_3]$ be irreducible factors of $G$ of the form
\begin{align*}
f\,&=\alpha_1\beta_1 Y_3^i+\text{lower terms in}\ Y_3,\cr
h\,&=\alpha_2\beta_2 Y_3^j+\text{lower terms in}\ Y_3,
\end{align*}
where $\beta_1,\beta_2\in\overline\f_p[Y_1,Y_2]$. We first claim that $\sigma(f)=f$. Otherwise, $f\sigma(f)\mid G$. Since $\sigma(\alpha_1)=\alpha_1$, it follows that $\alpha_1^2\mid a_8$, which is a contradiction. In the same way $\sigma(h)=h$.

Next, we claim that either $f$ or $h$ is cyclic. Otherwise, $ff^\rho f^{\rho^2}\mid G$ and $hh^\rho h^{\rho^2}\mid G$. Then $\deg f\le 18/3=6$ and $\deg h\le 6$. We must have $h\in\{f,f^\rho,f^{\rho^2}\}$. (Otherwise, $ff^\rho f^{\rho^2}hh^\rho h^{\rho^2}\mid G$, whence $\deg G\ge 6\cdot 4>18$, which is a contradiction.) If $h=f$, then $\alpha_2\mid \beta_1$. Hence $\deg f\ge \deg(\alpha_1\alpha_2)=8$, which is a contradiction. If $h=f^\rho$, then $\deg h\ge 4+\deg_{Y_3}h=4+\deg_{Y_2}f\ge 7$, which is a contradiction. If $h=f^{\rho^2}$, i.e., $f=h^\rho$, we have $\deg f\ge 7$, which is also a contradiction.
\end{proof}


\section{The Case $n=4$}

The proof for the case $n=4$ is more complicated but is based on the same approach for the case $n=3$.

\begin{thm}\label{T4.1} Let $p\ge 100,018,663$ be a prime and $b\in \f_{p^4}$ be such that $\text{\rm Tr}_{p^4/p}(b)\ne 0$. Then $f_b$ is not a PR of $\f_{p^4}$. (Note: $100,018,663$ is the $5,762,458$th prime.)
\end{thm}

\begin{proof}
First, by \cite[Conjecture 4.1$'$]{Hou-Sze-arXiv:1910.11989}, it suffices to show that $f_{1/2}$ is not a PR of $\f_{p^4}$. We have
\[
f_{1/2}(X+Y)-f_{1/2}(X)=\frac{YF(X,Y)}{Z(X)Z(X+Y)},
\]
where
\[
Z(X)=X^p-X+\frac 12
\]
and
\begin{equation}\label{4.1}
F(X,Y)=Z(X)^2+(Y^p-Y)Z(X)+1-Y^{p-1}\in\f_p[X,Y].
\end{equation}
It suffices to show that there exists $(x,y)\in\f_{p^4}^2$ with $y\ne 0$ such that $F(x,y)=0$. To this end, it suffices show that there exist $\Delta,y\in\f_{p^4}$, with $y\ne 0$ and $\text{Tr}_{p^4/p}(\Delta)=2\text{Tr}_{p^4/p}(1/2)=4$, satisfying 
\begin{equation}\label{4.2}
\Delta^2=y^{2p}+y^2-2y^{1+p}+4y^{p-1}-4;
\end{equation}
see \eqref{3.1} -- \eqref{3.4}. 

Assume that such $\Delta$ and $y$ exist, and let $y_i=y^{p^{i-1}}$ and $\Delta_i=\Delta^{p^{i-1}}$, $1\le i\le 4$. Then
\begin{equation}\label{4.3}
\Delta_i^2=y_i^2+y_{i+1}^2-2y_iy_{i+1}+4\frac{y_{i+1}}{y_i}-4,\quad 1\le i\le 4,
\end{equation}
where the subscript is taken modulo 4, and
\begin{equation}\label{4.4}
\Delta_1+\Delta_2+\Delta_3+\Delta_4=4.
\end{equation}
Under the condition \eqref{4.4}, one can express $\Delta_1$ in terms of $\Delta_1^2,\dots,\Delta_4^2$ as follows:
\begin{align*}
&(4-\Delta_1-\Delta_2)^2=(\Delta_3+\Delta_4)^2,\cr
&(4-\Delta_1)^2+\Delta_2^2-2(4-\Delta_1)\Delta_2=\Delta_3^2+\Delta_4^2+2\Delta_3\Delta_4,\cr
&(16+\Delta_1^2-8\Delta_1+\Delta_2^2-\Delta_3^2-\Delta_4^2)^2=\bigl[2(4-\Delta_1)\Delta_2+2\Delta_3\Delta_4\bigr]^2,\cr
&(16+\Delta_1^2+\Delta_2^2-\Delta_3^2-\Delta_4^2)^2+64\Delta_1^2-16\Delta_1(16+\Delta_1^2+\Delta_2^2-\Delta_3^2-\Delta_4^2)\cr
&=4\bigl[(4-\Delta_1)^2\Delta_2^2+\Delta_3^2\Delta_4^2+2(4-\Delta_1)\Delta_2\Delta_3\Delta_4\bigr],\cr
&(16+\Delta_1^2+\Delta_2^2-\Delta_3^2-\Delta_4^2)^2+64\Delta_1^2-16\Delta_1(16+\Delta_1^2+\Delta_2^2-\Delta_3^2-\Delta_4^2)\cr
&-4(4-\Delta_1)^2\Delta_2^2-4\Delta_3^2\Delta_4^2=8(4-\Delta_1)\Delta_2\Delta_3\Delta_4.
\end{align*}
Squaring both sides leads to
\begin{equation}\label{4.5}
\Delta_1=\frac{A(\Delta_1^2,\Delta_2^2,\Delta_3^2,\Delta_4^2)}{32B(\Delta_1^2,\Delta_2^2,\Delta_3^2,\Delta_4^2)},
\end{equation}
provided $B(\Delta_1^2,\Delta_2^2,\Delta_3^2,\Delta_4^2)\ne 0$, where
\begin{align*}
A(X_1,X_2,X_3,X_4)&=65536+114688X_1+\cdots+X_4^4,\cr
B(X_1,X_2,X_3,X_4)&=4096+1792X_1+\cdots-X_4^3.
\end{align*}
In the same way,
\begin{equation}\label{4.6}
\Delta_i=\frac{A(\Delta_i^2,\Delta_{i+1}^2,\Delta_{i+2}^2,\Delta_{i+3}^2)}{32B(\Delta_i^2,\Delta_{i+1}^2,\Delta_{i+2}^2,\Delta_{i+3}^2)},\quad 1\le i\le 4.
\end{equation}
The equation
\begin{equation}\label{4.7}
\Bigl[\frac{A(\Delta_1^2,\Delta_2^2,\Delta_3^2,\Delta_4^2)}{32B(\Delta_1^2,\Delta_2^2,\Delta_3^2,\Delta_4^2)}\Bigr]^2=\Delta_1^2
\end{equation}
can be written as 
\begin{equation}\label{4.8}
\frac{P(\Delta_1^2,\Delta_2^2,\Delta_3^2,\Delta_4^2)}{1024B(\Delta_1^2,\Delta_2^2,\Delta_3^2,\Delta_4^2)^2}=0,
\end{equation}
where
\[
P(X_1,X_2,X_3,X_4)=4294967296-2147483648X_1+\cdots+X_4^8.
\]
Using \eqref{4.6}, equation \eqref{4.4} becomes
\begin{equation}\label{4.9}
\frac{P(\Delta_1^2,\Delta_2^2,\Delta_3^2,\Delta_4^2)Q(\Delta_1^2,\Delta_2^2,\Delta_3^2,\Delta_4^2)}{16\prod_{i=1}^4B(\Delta_i^2,\Delta_{i+1}^2,\Delta_{i+2}^2,\Delta_{i+3}^2)^2}=0,
\end{equation}
where
\[
Q(X_1,X_2,X_3,X_4)=-209715-196608\Delta_1^2+\cdots+\Delta_4^5.
\]
Using \eqref{4.3}, we can write
\begin{align*}
P(\Delta_1^2,\Delta_2^2,\Delta_3^2,\Delta_4^2)=\,&-\frac{2^{16}}{(y_1y_2y_3y_4)^8}G(y_1,y_2,y_3,y_4),\cr
B(\Delta_1^2,\Delta_2^2,\Delta_3^2,\Delta_4^2)=\,&-\frac{2^4}{(y_1y_2y_3y_4)^3}L(y_1,y_2,y_3,y_4),
\end{align*}
where
\[
G(Y_1,Y_2,Y_3,Y_4)=-Y_1^{16}Y_2^8Y_3^8+16Y_1^{15}Y_2^8Y_3^8Y_4+\cdots+4Y_1^8Y_2^{10}Y_3^{12}Y_4^{16}
\]
and 
\[
L(Y_1,Y_2,Y_3,Y_4)=4Y_1^6Y_2^3Y_3^3-10Y_1^4Y_2^3Y_3^3Y_4-\cdots+Y_1^3Y_2^3Y_3^5Y_4^7
\]
are cyclic polynomials over $\f_p$ of degree 46 and 18, respectively, and $\deg_{Y_i}G=16$, $1\le i\le 4$. More precisely,
\begin{equation}\label{4.9.1}
G=g_{46}+g_{44}+g_{42}+g_{40}+g_{38}+g_{36}+g_{34}+g_{32},
\end{equation}
where $g_i\in\f_p[Y_1,Y_2,Y_3,Y_4]$ is homogeneous of degree $i$:
\begin{align}\label{4.10}
g_{46}=\,&-4(Y_1Y_2Y_3Y_4)^8\bigl[(Y_1-Y_2)(Y_2-Y_3)(Y_3-Y_4)(Y_4-Y_1)\bigr]^2\\ 
&\cdot\bigl[(Y_1-Y_3)(Y_2-Y_4)\bigr]^2(Y_1-Y_2+Y_3-Y_4)^2,\cr
\label{4.11}
g_{44}=\,&(Y_1Y_2Y_3Y_4)^6(Y_1^{10}Y_2^6Y_3^4-4Y_1^9Y_2^7Y_3^4+\cdots+Y_2^4Y_3^6Y_4^{10}),\\
\label{4.12}
g_{42}=\,&2(Y_1Y_2Y_3Y_4)^5(Y_1^{11}Y_2^7Y_3^4-4Y_1^{10}Y_2^8Y_3^4+\cdots-Y_1^2Y_2^2Y_3^7Y_4^{11}),\\
\label{4.13}
g_{40}=\,&(Y_1Y_2Y_3Y_4)^4(Y_1^{12}Y_2^8Y_3^4-4Y_1^{11}Y_2^9Y_3^4+\cdots+Y_1^4Y_3^8Y_4^{12}),\\
\label{4.14}
g_{38}=\,&2(Y_1Y_2Y_3Y_4)^3(2Y_1^{13}Y_2^8Y_3^5-6Y_1^{12}Y_2^9Y_3^5+\cdots-2Y_1^5Y_2^2Y_3^6Y_4^{13}),\\
\label{4.15}
g_{36}=\,&2(Y_1Y_2Y_3Y_4)^2(3Y_1^{14}Y_2^8Y_3^6-6Y_1^{13}Y_2^9Y_3^6+\cdots+3Y_1^6Y_2^4Y_3^4Y_4^{14}),\\
\label{4.16}
g_{34}=\,&4Y_1Y_2Y_3Y_4(Y_1^{15}Y_2^8Y_3^7-Y_1^{14}Y_2^9Y_3^7+\cdots-Y_1^7Y_2^6Y_3^2Y_4^{15}),\\
\label{4.17}
g_{32}=\,&Y_1^{16}Y_2^8Y_3^8-16Y_1^{15}Y_2^8Y_3^8Y_4-\cdots+Y_1^8Y_2^8Y_4^{16}.
\end{align}
Later, we will use the fact that
\begin{equation}\label{4.18}
\text{gcd}(G,L)=1.
\end{equation}
This fact follows from the computation that 
\begin{align*}
&\text{Res}(G(1,1,Y_3,Y_4),L(1,1,Y_3,Y_4);Y_4)\cr
&=2^{112}
(-1+Y_3)^8Y_3^{56} (3+Y_3)^8 (-19+2Y_3+Y_3^2)^8 (-1+8Y_3-26Y_3^2 +3Y_3^4)^4 \cr
&\kern1em \cdot (-16+11Y_3+88Y_3^2 -16Y_3^3 -98Y_3^4 +5Y_3^5 +10Y_3^6)^4 \cr
&\kern1em \cdot  (16+8Y_3-204Y_3^2 +369Y_3^3 +30Y_3^4 -76Y_3^5 -2Y_3^6 +3Y_3^7)^4\cr
&\kern1em \cdot  (256+672Y_3+505Y_3^2 -4456Y_3^3 +5718Y_3^4 -364Y_3^5 -1139Y_3^6\cr       &\kern1.7em +52Y_3^7 +52Y_3^8)^2(-256-256Y_3+832Y_3^2 +672Y_3^3 -1056Y_3^4 \cr
&\kern1.7em  -392Y_3^5 +431Y_3^6 +76Y_3^7 -66Y_3^8 -4Y_3^9 +3Y_3^{10})^2\cr
&\ne 0.
\end{align*}

To prove the theorem, it suffice to show that there exists $y\in\f_{p^4}$ such that $G(y,y^p,y^{p^2},y^{p^3})=0$ but $L(y,y^p,y^{p^2},y^{p^3})\ne 0$. Once this is done, the proof of the theorem is completed as follows: Clearly, $y\ne 0$. Let $y_i=y^{p^{i-1}}$, $1\le i\le 4$. Let $\Delta_i$ ($1\le i\le 4$) be given by \eqref{4.6} and in \eqref{4.6} let $\Delta_i^2$ be given by \eqref{4.3}, whence $\Delta_i$ is defined in terms of $y_1,\dots,y_4$. For $\Delta_i$ so defined,  \eqref{4.8} and \eqref{4.9} are satisfied. Let $\Delta=\Delta_1$. From \eqref{4.8} we have \eqref{4.7} and hence \eqref{4.2}; from \eqref{4.9} we have \eqref{4.4}, i.e., $\text{Tr}_{p^4/p}(\Delta)=4$.

Choose $z\in\f_{p^4}$ such that $M(z)$ is invertible and let 
\[
G_1=G((Y_1,Y_2,Y_3,Y_4)M(z)).
\]
By Lemma~\ref{L4.2} below, $G$ has a cyclic absolutely irreducible factor $d\in\f_p[Y_1,Y_2,Y_3,Y_4]$. Then $d_1=d((Y_1,Y_2,Y_3,Y_4)M(z))\in\f_p[Y_1,Y_2,Y_3,Y_4]$ is a cyclic absolutely irreducible factor of $G_1$. By \cite[Theorem~5.2]{Cafure-Matera-FFA-2006},
\begin{align*}
|V_{\f_p^4}(d_1)|\,&\ge p^2-(46-1)(46-2)p^{3/2}-5\cdot 46^{13/3}p\cr
&=p^2-1980p^{3/2}-5\cdot 46^{13/3}p.
\end{align*}
Let $L_1=L((Y_1,Y_2,Y_3,Y_4)M(z))$. By \eqref{4.18}, $\text{gcd}(G_1,L_1)=1$. Then by \cite[Lemma~2.2]{Cafure-Matera-FFA-2006},
\[
|V_{\f_p^4}(G_1)\cap V_{\f_p^4}(L_1)|\le 46^2p.
\]
Hence 
\begin{align*}
|V_{\f_p^4}(G_1)\setminus V_{\f_p^4}(L_1)|\,&\ge |V_{\f_p^4}(d_1)|-|V_{\f_p^4}(G_1)\cap V_{\f_p^4}(L_1)|\cr
&\ge p^2-1980p^{3/2}-(5\cdot 46^{13/3}+46^2)p\cr
&=p( p-1980p^{1/2}-(5\cdot 46^{13/3}+46^2))\cr
&>0
\end{align*}
since 
\[
p\ge 100,018,663>100,018,659\approx\frac 14\bigl[1980+(1980^2+4(5\cdot 46^{13/3}+46^2))^{1/2}\bigr]^2.
\]
Let $(a_1,a_2,a_3,a_4)\in V_{\f_p^4}(G_1)\setminus V_{\f_p^4}(L_1)$ and let $y=a_1z+a_2z^p+a_3z^{p^2}+a_4z^{p^3}\in\f_{p^4}$. Then $G(y,y^p,y^{p^2},y^{p^3})=0$ but $L(y,y^p,y^{p^2},y^{p^3})\ne 0$.
\end{proof}

\begin{lem}\label{L4.2}
The polynomial $G$ in \eqref{4.9.1} has a cyclic absolutely irreducible factor in $\f_p[Y_1,Y_2,Y_3,Y_4]$. \end{lem}

\begin{proof}
Throughout this proof, $\rho$ denotes the cyclic shift $(Y_1,Y_2,Y_3,Y_4)\mapsto(Y_2,Y_3,Y_4,Y_1)$ and $\sigma\in\text{Aut}(\overline\f_p)$ is the Frobenius map $(\ )\mapsto(\ )^p$. For any $f\in\overline\f_p[Y_1,Y_2,Y_3,Y_4]$, $f_i$ denotes the homogenous component of degree $i$ in $f$. For $f=f_i+f_{i-1}+\cdots$ with $f_i\ne 0$, let $\bar f=f_i+f_{i-1}X+\cdots\in\overline\f_p[Y_1,Y_2,Y_3,Y_4,X]$ be the homogenization of $f$ and let $\tilde f=\bar f(Y_1,Y_2,Y_3,Y_4,-1)=f_i-f_{i-1}+f_{i-2}-\cdots$. Since $\overline G=g_{46}+g_{44}X^2+\cdots+g_{32}X^{14}$, we have $\widetilde G=G$. Therefore, if $f\mid G$, then $\bar f\mid\overline G$, whence $\tilde f\mid\widetilde G$, i.e., $\tilde f\mid G$.

\medskip
$1^\circ$ Let $y_1,y_2,y_3$ be independent indeterminates and let $y_4$ be a root of $G(y_1,y_2,y_3,Y_4)$. We claim that $4\mid [\overline\f_p(y_1,y_2,y_3,y_4):\overline\f_p(y_1,y_2,y_3)]$. Let $\Delta_i$, in terms of $y_1,\dots,y_4$, be given by \eqref{4.6} and \eqref{4.3}. Then \eqref{4.3} is satisfied. By \eqref{4.3}, $\Delta_1^2,\Delta_2^2\in\f_p(y_1,y_2,y_3)$ and it is easy to see that $\Delta_1^2$, $\Delta_2^2$ and $\Delta_1^2\Delta_2^2$ are all nonsquares in $\overline\f_p(y_1,y_2,y_3)$. It follows that 
\[
[\overline\f_p(y_1,y_2,\Delta_1,\Delta_2):\overline\f_p(y_1,y_2,y_3)]=4.
\]
Hence the claim. (A more general fact which is easily proved by induction: If $F$ is a field with $\text{char}\,F\ne 2$ and $u_1,\dots, u_n\in F$ are such that for every $\emptyset\ne I\subset\{1,\dots,n\}$, $\prod_{i\in I}u_i$ is a nonsquare in $F$, then $[F(\sqrt{u_1},\dots,\sqrt{u_n}):F]=2^n$ and $\text{Aut}(F(\sqrt{u_1},\dots,\sqrt{u_n})/F)\cong(\Bbb Z/2\Bbb Z)^n$.)

\[
\beginpicture
\setcoordinatesystem units <4mm,4mm> point at 0 0

\setlinear
\plot 0 1  0 3 /
\plot 0 5  0 7 /

\put {$\overline\f_p(y_1,y_2,y_3)$} at 0 0
\put {$\overline\f_p(y_1,y_2,y_3,\Delta_1,\Delta_2)$} at 0 4
\put {$\overline\f_p(y_1,y_2,y_3,y_4)$} at 0 8
\put {$\scriptstyle 4$} [l] at 0.3 2

\endpicture
\]
\medskip


$2^\circ$ We claim that $g_{32}$ has no factors with multiplicity $>1$. This claim follows from the following computation:
\begin{align*}
&\text{Res}\Bigl( g_{32}(1,-1,Y_3,Y_4),\,\frac{\partial}{\partial Y_4}g_{32}(1,-1,Y_3,Y_4);\,Y_4\Bigr)\cr
&=2^{256}\cdot 3^8\,Y_3^{148}(1-Y_3)^8(1+Y_3)^8\cdots(65536+245760Y_3-\cdots+Y_3^{12})\cr
&\ne 0.
\end{align*}

\medskip
$3^\circ$ We claim that if $f\in\overline\f_p[Y_1,Y_2,Y_3,Y_4]$ is an irreducible factor of $G$, then $\tilde f=f$. Assume the contrary. Then $f\tilde f\mid G$. Write $f=f_i+f_{i-1}+\cdots+f_{i-s}$, where $f_if_{i-s}\ne 0$. By $1^\circ$, $i\ge 4$. Note that $\tilde f=f_i-f_{i-1}+\cdots+ (-1)^sf_{i-s}$. It follows that $f_{i-s}^2\mid g_{32}$. By $2^\circ$, we have $i-s=0$, that is, $f=1+f_1+\cdots+f_i$. We have
\[
G(Y_1,Y_2,Y_1,Y_2)=-2^8(Y_1Y_2)^{12}\alpha(Y_1,Y_2)\alpha(Y_2,Y_1)\beta(Y_1,Y_2)^2\beta(Y_2,Y_1)^2,
\]
where 
\begin{align*}
\alpha(Y_1,Y_2)\,&=Y_1^2-Y_1Y_2+Y_1^2Y_2+Y_2^2-Y_1Y_2^2,\cr
\beta(Y_1,Y_2)\,&=4Y_1-8Y_2+Y_1^2Y_2-2Y_1Y_2^2+Y_2^3.
\end{align*}
It is easy to see see that $\alpha$ and $\beta$ are irreducible in $\overline\f_p[Y_1,Y_2]$. (The discriminants of $\alpha$ and $\beta$, as polynomials in $Y_1$, are $Y_2^2(-3+Y_2)(1+Y_2)$ and $16(1+Y_2^2)$, respectively. These are  nonsquares in $\overline\f_p(Y_2)$.) Therefore $G(Y_1,Y_2,Y_1,Y_2)$ does not have any nonconstant factor with nonzero constant term. It follows that $f(Y_1,Y_2,Y_1,Y_2)=1$. Then
\[
42=\deg G(Y_1,Y_2,Y_1,Y_2)\le \deg(G/f\tilde f)\le 46-2i\le 46-8=38,
\]
which is a contradiction.

\medskip
$4^\circ$ We claim that if $f$ is a pseudo-cyclic absolutely irreducible factor of $G$, then $f$ is cyclic. We have $f^\rho=cf$, where $c\in\overline\f_p^*$. Let $h=G/f$. By $3^\circ$, $f=f_i+f_{i-2}+\cdots$ ($f_i\ne 0$) and hence $h=h_j+h_{j-2}+\cdots$ ($h_j\ne 0$). Then $f_ih_j=g_{46}$. Since $f_i\mid g_{46}$ and $f_i^\rho=cf_i$, it follows from \eqref{4.10} that 
\begin{align*}
f_i=\,&d(Y_1Y_2Y_3Y_4)^{i_1}\bigl[(Y_1-Y_2)(Y_2-Y_3)(Y_3-Y_4)(Y_4-Y_1)\bigr]^{i_2}\cr 
&\cdot\bigl[(Y_1-Y_3)(Y_2-Y_4)\bigr]^{i_3}(Y_1-Y_2+Y_3-Y_4)^{i_4},
\end{align*}
where $d\in\overline \f_p^*$, $0\le i_1\le 8$, $0\le i_2,i_3,i_4\le 2$. Thus $f_i^\rho=\pm f_i$. If  $f_i^\rho=-f_i$, it follows from $f_ih_j= g_{46}$ that either $(Y_1-Y_3)(Y_2-Y_4)$ or $Y_1-Y_2+Y_3-Y_4$ divides both $f_i$ and $h_j$. Then $g_{44}=f_ih_{j-2}+f_{i-2}h_j$ is divisible by $(Y_1-Y_3)(Y_2-Y_4)$ or $Y_1-Y_2+Y_3-Y_4$ , which is a contradiction.

\medskip
$5^\circ$
We claim that $G$ cannot be written as $G=cff'hh'$, where $c\in\overline \f_p^*$, $f,g\in\overline \f_p[Y_1,Y_2,Y_3,Y_4]$ are irreducible or equal to $1$. $f'=f^\rho$ or $\sigma(f)$, and $h'=h^\rho$ or $\sigma(h)$. Otherwise, by $3^\circ$, $f=f_i+f_{i-2}+\cdots+f_{i-2s}$ and $h=h_j+h_{j-2}+\cdots+h_{j-2t}$, where $f_if_{i-2s}h_jh_{j-2t}\ne 0$. Since $G=g_{46}+\cdots+g_{32}$, we have $2(i+j)=46$ and $2(i-2s+j-2t)=32$, which is impossible.

\medskip
$6^\circ$ Let
\[
k=\min\bigl\{\deg_{Y_i}f:f\in\overline\f_p[Y_1,Y_2,Y_3,Y_4]\ \text{is irreducible},\ f\mid G,\ 1\le i\le 4\Bigr\}.
\]
We may assume that $k=\deg_{Y_4}f$ for some irreducible factor $f$ of $G$ in $\overline\f_p[Y_1,Y_2,Y_3,Y_4]$. Clearly, $G$ does not have any nontrivial factor in $\overline\f_p^[Y_1,Y_2,Y_3]$, so $k>0$. By $1^\circ$, we have $k\in\{4,8,16\}$. Let $l$ be the smallest integer such that $f\in\f_{p^l}[Y_1,Y_2,Y_3,Y_4]$.

\medskip
{\bf Case 1.} Assume that $k=16$. Then $G=f$ and we are done.

\medskip
{\bf Case 2.} Assume that $k=8$. We claim that $f$ cyclic. Otherwise, by $4^\circ$, $f$ is not pseudo-cyclic. Then $ff^\rho\mid G$. Since $\deg_{Y_4}(ff^\rho)\ge 8+8=16$, we have $G=cff^\rho$ for some $c\in\overline\f_p^*$, which is impossible by $5^\circ$. Hence the claim is proved.

If $l>1$, then $G=d\,f\sigma(f)$ for some $d\in\overline\f_p^*$, which is impossible by $5^\circ$. Hence $l=1$, i.e., $f\in\f_p[Y_1,Y_2,Y_3,Y_4]$.

\medskip
{\bf Case 3.}
Assume that $k=4$. We first claim that $f^{\rho^2}=cf$ for some $c\in\overline\f_p^*$. Otherwise, $ff^\rho f^{\rho^2}f^{\rho^3}\mid G$. Since $\deg_{Y_4}(ff^\rho f^{\rho^2}f^{\rho^3})\ge 4\cdot 4=\deg_{Y_4}G$, we have $G=d\,ff^\rho f^{\rho^2}f^{\rho^3}$ for some $d\in\overline\f_p^*$, which is impossible by $5^\circ$. So the claim is proved. Write $f=f_i+f_{i-2}+\cdots$,
where $f_i\ne 0$. Since $f_i\mid g_{46}$ and $f_i^{\rho^2}=cf_i$, we have $c=1$, so $f^{\rho^2}=f$.

\medskip
{\bf Case 3.1.} Assume that $f$ is cyclic. Since $f\sigma(f)\cdots\sigma^{l-1}(f)\mid G$, we have $4l\le \deg_{Y_4}G=16$, i.e., $l\le 4$.

If $l=1$, then $f$ is a cyclic absolutely irreducible factor of $G$ in $\f_p[Y_1,Y_2,Y_3,Y_4]$, and we are done.

If $l=4$, then $G=ef\sigma(f)\sigma^2(f)\sigma^3(f)$ for some $e\in\overline\f_p^*$, which is impossible by $5^\circ$.

If $l=3$, $G/f\sigma(f)\sigma^2(f)$ is a cyclic absolutely irreducible factor of $G$ in $\f_p[Y_1,Y_2,Y_3,Y_4]$, and we are done.

If $l=2$, then $H:=G/f\sigma(f)$ is cyclic and belongs to $\f_p[Y_1,Y_2,Y_3,Y_4]$. If $H$ is absolutely irreducible, we are done. So assume that $H$ has a proper absolutely irreducible factor $h$. By the minimality of $k$, we have $\deg_{Y_4}h=4$. If $h$ is not pseudo-cyclic, then $hh^\rho\mid H$, whence $G=\epsilon f\sigma(f)hh^\rho$ for some $\epsilon\in\overline \f_p^*$, which is impossible by $5^\circ$. Therefore $h$ is pseudo-cyclic and hence cyclic (by $4^\circ$). If $\sigma(h)/h$ is not a constant, we have $h\sigma(h)\mid H$, which leads to the same contradiction. Thus $\sigma(h)/h$ is a constant. We may assume that $\sigma(h)=h$. Now $h$ is a cyclic absolutely irreducible factor of $G$ in $\f_p[Y_1,Y_2,Y_3,Y_4]$.  

\medskip
{\bf Case 3.2.} Assume that $f$ is not cyclic. By $4^\circ$, $f$ is not pseudo-cyclic. Then $ff^\rho$ is a cyclic factor of $G$. We claim that $\sigma(ff^\rho)/ff^\rho$ is a constant. (Otherwise, $f$, $f^\rho$, $\sigma(f)$ and $\sigma(f^\rho)$ are different factors of $G$. Then $G=e ff^\rho\sigma(f)\sigma(f^\rho)$ for some $e\in\overline\f_p^*$, which is impossible by $5^\circ$.)  We may assume that $ff^\rho\in\f_p[Y_1,Y_2,Y_3,Y_4]$. Let $H=G/ff^\rho$, which is cyclic and belongs to $\f_p[Y_1,Y_2,Y_3,Y_4]$. By $5^\circ$, $H$ is not a constant. Let $h$ be an absolutely irreducible factor of $H$. By the minimality of $k$, $\deg_{Y_4}h\ge 4$. If $h$ is not cyclic, then $hh^\rho\mid H$. Thus $G=\epsilon ff^\rho hh^\rho$ for some $\epsilon\in\overline\f_p^*$, which is impossible by $5^\circ$. So $h$ is cyclic. If $\sigma(h)/h$ is not a constant, then $h\sigma(h)\mid H$, which leads to the same contradiction. So $\sigma(h)/h$ is a constant, and we may assume that $\sigma(h)=h$. Now $h$ is a cyclic absolutely irreducible factor of $G$ in $\f_p[Y_1,Y_2,Y_3,Y_4]$.

\medskip
The proof of the lemma is now complete.
\end{proof}

\section*{Acknowledgment}

The research of D. Bartoli was supported by the Italian National Group for Algebraic and Geometric Structures and their Applications (GNSAGA - INdAM).




\begin{thebibliography}{99}

\bibitem{Bartoli-FFA-2018}
D. Bartoli, {\it On a conjecture about a class of permutation trinomials}, Finite Fields Appl. {\bf 52} (2018), 30 -- 50.

\bibitem{Bartoli-FFA-2020}
D. Bartoli, {\it Permutation trinomials over $\Bbb F_{q^3}$}, Finite Fields Appl.  {\bf 61} (2020), Article 101597.

\bibitem{Bartoli-Giulietti-FFA-2018}
D. Bartoli and M. Giulietti, 
{\it Permutation polynomials, fractional polynomials, and algebraic curves}, Finite Fields Appl. {\bf 51} (2018), 1 -- 16. 

\bibitem{Cafure-Matera-FFA-2006}
A. Cafure and G. Matera, {\it Improved explicit estimates on the number of solutions of equations over a finite field}, Finite Fields Appl.  {\bf 12} (2006), 155 -- 185.

\bibitem{Cao-Hou-Mi-Xu-FFA-2020}
X. Cao, X. Hou, J. Mi, S. Xu,
{\it More permutation polynomials with Niho exponents which permute $\Bbb F_{p^2}$}, Finite Fields Appl.  {\bf 62} (2020), Article 101626.

\bibitem{Hou-CC-2019}
X. Hou, {\it On a class of permutation trinomials in characteristic 2}, Cryptography and Communications, {\bf 11} (2019), 1199 -- 1210.

\bibitem{Hou-arXiv1906.07240}
X. Hou, {\it On the Tu-Zeng Permutation Trinomial of Type $(1/4,3/4)$}, arXiv:1906.07240.

\bibitem{Hou-Sze-arXiv:1910.11989} 
X. Hou and Sze, {\it On a type of permutation rational functions over finite fields}, arXiv:1910.11989.

\bibitem{Hou-Tu-Zeng-FFA-2020}
X. Hou, Z. Tu, X. Zeng, {\it Determination of a class of permutation trinomials in characteristic three}, Finite Fields Appl. {\bf 61} (2020),  Article 101596.

\bibitem{Lang-Weil-AJM-1954}
S. Lang and A. Weil, {\it Number of points of varieties in finite fields}, 
Amer. J. Math.  {\bf 76} (1954), 819 -- 827. 

\bibitem{Li-Qu-Li-Fu-AA-2018}
K. Li, L. Qu, C. Li, S. Fu, {\it New permutation trinomials constructed from fractional polynomials}, Acta Arith. {\bf 183} (2018),  101 -- 116.

\bibitem{Tu-Zeng-FFA-2018}
Z. Tu and X. Zeng, {\it Two classes of permutation trinomials with Niho exponents}, Finite Fields Appl. {\bf 53} (2018), 99 -- 112.

\bibitem{Tu-Zeng-Li-Helleseth-FFA-2018}
Z. Tu, X. Zeng, C. Li, T. Helleseth, {\it A class of new permutation trinomials}, Finite Fields Appl. {\bf 50} (2018), 178 -- 195.

\bibitem{Wang-2019}
Q. Wang, {\it Polynomials over finite fields: an index approach}, In: Combinatorics and Finite Fields, Editors: K-U. Schmidt and A. Winterhof,
De Gruyter, 2019, pp. 319 -- 348.  

\bibitem{Yuan-Ding-Wang-Pieprzyk-FFA-2008}
J. Yuan, C. Ding, H. Wang, J. Pieprzyk,
{\it Permutation polynomials of the form $(x^p-x+\delta)^s+L(x)$},  Finite Fields Appl. {\bf 14} (2008), 482 -- 493.

\bibitem{Zieve-PAMS-2009}
M. E. Zieve, {\it On some permutation polynomials over $\Bbb F_q$  of the form $x^r h(x^{(q-1)/d})$}, Proc. Amer. Math. Soc. {\bf 137} (2009), 2209 -- 2216.



\end{thebibliography}
\end{document}